\documentclass
[a4paper,12pt]
{amsart}

\usepackage{graphicx}
\usepackage{amssymb,amsmath,amsfonts, amsthm}
\usepackage{amscd,enumerate}
\usepackage{times}
\usepackage{color}
\newtheorem{theorem}{Theorem}[section]

\newtheorem{lemma}[theorem]{Lemma}

\newtheorem{problem}[theorem]{Problem}

\newtheorem{corollary}[theorem]{Corollary}

\theoremstyle{remark}
\newtheorem{remark}[theorem]{Remark}

\theoremstyle{definition}

\newcommand{\C}{\Bbb C}
\def\BC{\mathbb C}
\def\BT{\mathbb T}
\newcommand{\Z}{\Bbb Z}
\newcommand{\D}{\Delta}

\newcommand{\tr}{{\mathrm{tr}\,}}
\newcommand{\la}{\langle}
\newcommand{\ra}{\rangle}
\newcommand{\p}{\partial}

\newcommand{\SL}{\mathrm{SL}(2,\C)}

\numberwithin{equation}{section}

\begin{document}

\title{Twisted Alexander polynomials of torus links}

\author{Teruaki Kitano, Takayuki Morifuji and Anh T. Tran}

\begin{abstract}
In this paper we give an explicit formula for the twisted Alexander polynomial of any torus link 
and show that it is a locally constant function on the $\SL$-character variety. 
We also discuss similar things for the higher dimensional twisted Alexander polynomial 
and the Reidemeister torsion. 
\end{abstract}



\address{Department of Information Systems Science, 
Faculty of Engineering, 
Soka University, 
Tangi-cho 1-236, 
Hachioji, Tokyo 192-8577, Japan}
\email{kitano@soka.ac.jp}

\address{Department of Mathematics,
Hiyoshi Campus, Keio University, 
Yokohama 223-8521, Japan}
\email{morifuji@z8.keio.jp}

\address{Department of Mathematical Sciences, The University of Texas at Dallas, 
Richardson, TX 75080, USA}
\email{att140830@utdallas.edu}

\maketitle

\section{Introduction}\label{section:1}

The twisted Alexander polynomial is a generalization of the classical Alexander polynomial of a knot 
in the $3$-sphere $S^3$. 
It was first introduced by Lin \cite{Lin01-1} for  knot groups and 
by Wada \cite{Wada94-1} for finitely presentable groups which include link groups. 
Recently 
this polynomial invariant has been widely investigated by lots of authors and recognized as a powerful tool in 
low dimensional topology. As for recent development of this topic and related references, see 
the survey papers \cite{FV10-1}, \cite{Kitano15-1} and \cite{Morifuji15-1}. 

In this paper we consider twisted Alexander polynomials associated to 
irreducible $\SL$-representations of groups of knots or links in $S^3$. 
In particular we investigate the behavior of the twisted Alexander polynomial as a function on 
the $\SL$-character variety, that is, the set of conjugacy classes of irreducible 
$\SL$-representations. 
As for this kind of property, it is known that there is a hyperbolic knot such that the twisted 
Alexander polynomial varies continuously on its character variety (see \cite{GM03-1}). 
On the other hand, it is known that such kind of phenomenon never happen for any torus knot. 
Actually in our previous paper \cite{KtM12-1} we showed that 
every coefficient of the twisted Alexander polynomial of a torus knot is a locally constant 
function of the $\SL$-character variety. 
Furthermore 
we reproved a result of Johnson that the Reidemeister torsion of a 
torus knot is locally constant on the $\SL$-character variety (see \cite{Johnson}). 

The purpose of this paper is to generalize the above result to torus links in $S^3$ and 
moreover to discuss similar things for higher dimensional twisted Alexander polynomials. 
To this end we give an explicit formula for the (higher dimensional) twisted Alexander polynomial, 
which depends only on the eigenvalues of the $2\times2$ matrices corresponding to the cores of the 
two solid tori of the standard genus one Heegaard splitting of $S^3$. 
The property of the twisted Alexander polynomial and the Reidemeister torsion mentioned above 
immediately follows from the formula. 

This paper is organized as follows. In the next section, we quickly review the definition of 
the twisted Alexander polynomial of a link in $S^3$. 
In Section~\ref{section:3} we study the $\SL$-character variety of 
a torus link. 
In Section~\ref{section:4} we describe a formula for the 
twisted Alexander polynomial of a torus link. In the last section, 
we investigate the higher dimensional twisted Alexander polynomial and 
the Reidemeister torsion.

\section{Twisted Alexander polynomials}\label{section:2}

Let 
$L=L_1\sqcup\cdots\sqcup L_\mu$ be a $\mu$-component oriented link in $S^3$ 
and $E_L=S^3\setminus \mathrm{int}(N(L))$ the exterior of $L$ in $S^3$. 
Here $N(L)$ is a closed tubular neighborhood of $L$. 
We denote $\pi_1(E_L)$ by $G(L)$ and call it the link group. 
We choose and fix a Wirtinger presentation of $G(L)$: 
$$
G(L)
=
\la
x_1,\ldots,x_l\,|\,r_1,\ldots,r_{l-1}
\ra
$$
where every generator corresponds to an arc in a regular projection $D(L)$ of the link $L$ 
and every relator comes from a crossing in $D(L)$. 
The abelianization homomorphism 
$$
\alpha_L:G(L)\to H_1(E_L;\Z)\cong\Z^{\oplus\mu}
=
\la t_1\ra\oplus\cdots\oplus\la t_\mu \ra
$$
is given by assigning to each generator $x_i$ the meridian element 
$t_k\in H_1(E_L;\Z)$ of the corresponding component $L_k$ of $L$. 
Here we denote the sum in each $\Z$ multiplicatively.

In this paper 
we consider a representation of $G(L)$ into the two-dimensional special linear group 
$\SL$, say $\rho:G(L)\to \SL$. 
The maps $\rho$ and $\alpha_L$ naturally induce two ring
homomorphisms $\tilde{\rho}: {\Z}[G(L)] \rightarrow M(2,{\C})$ and
$\tilde{\alpha}_L:{\Z}[G(L)]\rightarrow {\Z}[t_1^{\pm1},\ldots,t_\mu^{\pm1}]$, 
where
${\Z}[G(L)]$ is the group ring of $G(L)$ and $M(2,{\C})$ is the
matrix algebra of degree $2$ over ${\C}$. Then
$\tilde{\rho}\otimes\tilde{\alpha}_L$ defines a ring homomorphism
${\Z}[G(L)]\to M\left(2,{\C}[t_1^{\pm1},\ldots,t_\mu^{\pm1}]\right)$. 
Let $F_l$ denote the
free group on generators $x_1,\ldots,x_l$ and
$$
\Phi:{\Z}[F_l]\to M\left(2,{\C}[t_1^{\pm1},\ldots,t_\mu^{\pm1}]
\right)
$$
the composition of the surjection
$\tilde{\phi}:{\Z}[F_l]\to{\Z}[G(L)]$
induced by the presentation of $G(L)$
and the map
$\tilde{\rho}\otimes\tilde{\alpha}_L:{\Z}[G(L)]\to M(2,{\C}[t_1^{\pm1},\ldots,t_\mu^{\pm1}])$.

Let us consider the $(l-1)\times l$ matrix $A$
whose $(i,j)$-entry is the $2\times 2$ matrix
$$
\Phi\left(\frac{\partial r_i}{\partial x_j}\right)
\in M\left(2,{\C}[t_1^{\pm1},\ldots,t_\mu^{\pm1}]\right),
$$
where
$\frac{\partial}{\partial x}:\Z [F_l]\to\Z [F_l]$
denotes the free differential. 
We call $A$ the \textit{Alexander matrix}\, of the link group $G(L)$ associated to $\rho$. 
For
$1\leq j\leq l$,
let us denote by $A_j$
the $(l-1)\times(l-1)$ matrix obtained from $A$
by removing the $j$th column.
We regard $A_j$ as
a $2(l-1)\times 2(l-1)$ matrix with coefficients in
${\C}[t_1^{\pm1},\ldots,t_\mu^{\pm1}]$. 
Then Wada's \textit{twisted Alexander polynomial}\/
\cite{Wada94-1} of a link $L$
associated to a representation $\rho:G(L)\to \SL$ is defined
to be the rational function
$$
\D_{L,\rho}(t_1,\ldots,t_\mu)
=\frac{\det A_j}{\det\Phi(x_j-1)}
$$
and well-defined up to multiplication by
$t_1^{2k_1}\cdots t_\mu^{2k_\mu}~(k_j\in{\Z})$. 
In particular, it does not depend on a choice of a presentation of $G(L)$. 

\begin{remark}\label{rmk:trivial-rep}
By definition, 
$\D_{L,\rho}(t_1,\ldots,t_\mu)$ is a rational function in the variables $t_1,\ldots,t_\mu$, 
but it will be a Laurent polynomial if $L$ is a link with two or more components~\cite[Proposition~9]{Wada94-1}, or 
$L$ is a knot $K$ and $\rho$ is non-abelian~\cite[Theorem~3.1]{KtM05-1}. 
We note that if $\rho,\rho':G(L)\to \SL$ are conjugate representations, 
then $\D_{L,\rho}(t_1,\ldots,t_\mu) = \D_{L,\rho'}(t_1,\ldots,t_\mu)$ 
holds (see \cite[Section~3]{Wada94-1}).
\end{remark}

\section{Character varieties of torus links}\label{section:3}

In this section, we discuss the $\SL$-character variety of the group of a torus link. 
To this end, we first review a presentation of the group of a torus link. 

\subsection{A presentation of the group of a torus link}\label{subsection:3.1}

Let $L=T(\mu p,\mu q)$ be a $\mu$-component torus link in $S^3$ where $p,q$ are coprime integers. 
The link group 
$G(L)=\pi_1(E_L)$ has the following presentation (see \cite[Lemma~2.2]{RZ87-1}): 
$$G(L) = 
\la m_1,\cdots,m_\mu, x, y,\ell 
\mid m_\mu m_{\mu-1}\cdots m_1=x^r y^s, [\ell, m_i] =1~(1\le i \le \mu) ,\ell=x^p = y^q\ra 
$$
where 
$x,y$ represent the cores of the two solid tori of the Heegaard splitting defined by 
the torus $T\subset S^3$ (namely $T$ determines the Heegaard decomposition of genus one in $S^3$), 
$\ell$ represents a parallel of the torus knot $T(p,q)$ on $T$, 
$m_i$ is a meridian of each component of $L$ and $r,s$ satisfy $ps+qr=1$. 
We note that $G(L)$ contains $G(p,q)=\la x,y\,|\,x^p=y^q\ra$, 
the group of the torus knot $T(p,q)$, as a subgroup. 

\begin{remark}\label{remark:baridge-number}
\begin{itemize}
\item[(1)]
The center of $G(L)$ is an infinite cyclic group $\Z$ generated by $x^p=y^q$. 

\item[(2)]
The abelianization homomorphism 
$\alpha_L:G(L)\rightarrow \la t_1\ra\oplus\cdots\oplus\la t_\mu\ra$ 
is given by 
\[
\alpha_L(m_i)=t_i~(1\le i \le \mu),~ \alpha_L(x)=t^q,~ \alpha_L(y)=t^p
\]
where we put $t=t_1\cdots t_\mu$. 
\item[(3)]
It is known that the bridge number of the torus link $T(\mu p,\mu q)$ is equal to $\min(\mu p,\mu q)$ 
(see \cite[Corollary~1.5]{RZ87-1}). 
\end{itemize}
\end{remark}

Using the relation $m_\mu m_{\mu-1}\cdots m_1=x^r y^s$, 
the above presentation of $G(L)$ can be reduced to 
$$
G(L) = 
\la m_1, \cdots,m_{\mu-1},x, y \mid [x^p,m_i] =1~(1\le i \le \mu-1), x^p = y^q \ra .
$$
The latter presentation will be useful for calculating the (higher dimensional) 
twisted Alexander polynomial in Sections~\ref{section:4} and \ref{section:higher}. 

\subsection{Character varieties of torus links}\label{subsection:2.2}
In this subsection we describe the character variety of $G(L)$ following the paper \cite{GM93-1}. 

Now let us consider an \textit{irreducible}\, representation $\rho:G(L)\to\SL$. 
Namely there is no nontrivial proper invariant subspace of $\C^2$ under the natural 
action of $\rho(G(L))$. A representation $\rho:G(L)\to\SL$ is called 
\textit{reducible}\, if it is not irreducible. 

Let $\mathcal{R}(\mu p,\mu q)$ be the set of irreducible $\SL$-representations 
of the link group $G(L)$ of a $\mu$-component torus link $L=T(\mu p,\mu q)$. 
We denote the quotient space by conjugation action of $\SL$ by 
$$
\mathcal{X}(\mu p,\mu q)
=\mathcal{R}(\mu p,\mu q)/\sim.
$$ 

\begin{remark}\label{rmk:character-variety}
We call $\mathcal{R}(\mu p,\mu q)$ the \textit{irreducible representation variety} of 
$L=T(\mu p,\mu q)$ and $\mathcal{X}(\mu p,\mu q)$ the \textit{irreducible character variety} 
of $L$. Actually, for a \textit{character} $\chi_\rho:G(L)\to \C$ defined by 
$\chi_\rho(\gamma)=\tr\rho(\gamma),\,\gamma\in G(L)$, 
it is known that  $\chi_\rho=\chi_{\rho'}\,(\rho,\rho'\in\mathcal{R}(\mu p,\mu q))$ if and only if 
$\rho,\rho'$ are conjugate (see \cite[Proposition~1.5.2]{CS83-1}).
\end{remark}

We also denote the $2\times 2$ identity matrix by $I$  
and the image of each generator of $G(L)$ by its capital letter. 

\begin{lemma}\label{lem:center}
If $\rho$ is irreducible, then $X^p=Y^q=\pm I$ hold. 
\end{lemma}
\begin{proof}
Assume that $X^p=Y^q\neq \pm I$. 
Since 
$X^p=Y^q$ is in the center of $\rho(G(L))$, 
$X^p=Y^q$ commutes with any one of $X,Y$ and $M_j~(j=1,\ldots,\mu-1)$.  
Thus an eigenvector of $X^p=Y^q$ is also an eigenvector of any of them. 
It contradicts the irreducibility of $\rho$. 
\end{proof}
  
By the above lemma, we can write the eigenvalues of $X$ and $Y$ 
as 
$\alpha(a)=\exp(\sqrt{-1}\pi a/p),~$
$0\leq a\leq p$ and 
$\beta(b)=\exp(\sqrt{-1}\pi b/q),~0\leq b\leq q$ respectively. 

\begin{remark}
Since $X^p=(-I)^a=Y^q=(-I)^b$, it holds that $a\equiv b~(\mathrm{mod}~ 2)$. 
\end{remark}

For $\rho\in \mathcal{R}(\mu p,\mu q)$, 
it will be defined by a set of matrices 
$(X,Y,M_1,\ldots,M_{\mu-1})$ which satisfy the relations coming from the presentation of $G(L)$. 
We note that the relations $[x^p,m_i]=1~(1\leq i\leq \mu-1)$ give no restriction to the matrices 
$(X,Y,M_1,\ldots,M_{\mu-1})$, because $X^p=Y^q=\pm I$ hold by Lemma~\ref{lem:center}. 

\medskip

\noindent
{\bf Notation.} We will use the following notations: 
$$
t_x=\tr(X), t_y=\tr(Y),t_i=\tr(M_i),
t_{xy}=\tr(XY),t_{xi}=\tr(XM_i), t_{xyi} = \tr(XYM_i)
$$
for $1 \le i \le \mu -1$.

\subsubsection{$\mu=2$}

We observe how a representative of an irreducible representation 
\[
\rho:G(L) = 
\la m_1, x, y 
\mid [x^p, m_1] =1,x^p = y^q\ra 
\rightarrow \SL
\]
in each conjugacy class can be determined by using traces. 
Here we fix two non negative integers $a,b~(0\leq a\leq p,~0\leq b\leq q,~a\equiv b\text{ mod }2)$ and then fix 
\[
\begin{split}
t_x&=\tr(X)=\alpha(a)+\alpha(a)^{-1}=2\cos(a\pi/p),\\
t_y&=\tr(Y)=\beta(b)+\beta(b)^{-1}=2\cos(b\pi/q).
\end{split}
\]

\noindent
{\bf Case 1.} $\rho|_{G(p,q)}$ is \textit{non-abelian}. 
By \cite[Proposition~4.3]{GM93-1}, up to conjugation, we can put  
$$
X=
\begin{pmatrix}
\alpha(a)&0\\
0&\alpha(a)^{-1}
\end{pmatrix},~t_x\not=\pm 2;~
Y=
\begin{pmatrix}
s&1\\
u&v
\end{pmatrix},~
u=sv-1,~ t_y\neq\pm 2.
$$
By $\alpha(a)^2-t_x\alpha(a)+1=0$, 
one sees that
\[
\alpha(a)=
\frac{t_x + \sqrt{t_x^2-4}}{2} \quad 
\text{ or } \quad 
\alpha(a)=
\frac{t_x -\sqrt{t_x^2-4}}{2}.
\]
Since $\begin{pmatrix}
\alpha(a)&0\\
0&\alpha(a)^{-1}
\end{pmatrix}$ 
is conjugate to 
$\begin{pmatrix}
\alpha(a)^{-1}&0\\
0&\alpha(a)
\end{pmatrix}$ 
by conjugation action of 
$\begin{pmatrix}
0&1\\
-1&0
\end{pmatrix}$, 
the matrix $X$ is perfectly determined by $t_x$ up to conjugation. 

Next we determine $Y$ by using $t_y, t_{xy}$. Since $s+v=t_y$, we have
$v=t_y-s$. 
To determine $s$, 
we use 
\[
\begin{split}
t_{xy}
&=\alpha(a) s+\alpha(a)^{-1}v\\
&=s(\alpha(a)-\alpha(a)^{-1})+\alpha(a)^{-1}t_y. 
\end{split}
\]
Thus we have 
\[
\begin{split}
s&=\frac{t_{x y}-\alpha(a)^{-1}t_y}{\alpha(a)-\alpha(a)^{-1}},\\
v&=t_y-s\\
&=-\frac{t_{x y}-\alpha(a) t_y}{\alpha(a)-\alpha(a)^{-1}},\\
u&=s v-1 \\
  &= - \frac{t_{xy}^2- t_{xy} t_x t_y + t_y^2 + t_x^2-4}{t_x^2-4}.
\end{split}
\]
Therefore $t_x,t_y,t_{xy}$ determine $X$ and $Y$ uniquely, up to conjugation. 

\begin{remark}
Note that $\alpha(a)\not=\alpha(a)^{-1}$ since $t_x\neq\pm 2$. Moreover, $u \not=0$ if and only if $t_{xy}^2- t_{xy} t_x t_y + t_y^2 + t_x^2-4 \not= 0$. 
\end{remark}


{\bf Case 1.1.} $u\neq 0$. This is equivalent to $t_{xy}^2- t_{xy} t_x t_y + t_y^2 + t_x^2-4  \not= 0$.  In this case $\rho|_{G(p,q)}$ is \textit{irreducible}. 

We 
now determine $M_1$ 
by using $t_1,t_{x1},t_{y1},t_{xy1}$. 
Put 
$M_1=
\begin{pmatrix}
\gamma&\delta\\
\varepsilon&\zeta
\end{pmatrix}
$ with $\gamma\zeta-\delta\varepsilon=1$ and $\gamma+\zeta=t_1$. 
Also we have
\[
\begin{split}
t_{x1}&=\tr(XM_1)\\
&=\alpha(a)\gamma+\alpha(a)^{-1}\zeta\\
&=\alpha(a)\gamma+\alpha(a)^{-1}(t_1-\gamma)\\
&=(\alpha(a)-\alpha(a)^{-1})\gamma+\alpha(a)^{-1}t_1,
\end{split}
\]
namely 
\[
\gamma=
\frac{t_{x1}-\alpha(a)^{-1}t_1}{\alpha(a)-\alpha(a)^{-1}}
\]
and 
\[
\begin{split}
\zeta
&=t_1-\gamma\\
&=-\frac{t_{x1}-\alpha(a)t_1}{\alpha(a)-\alpha(a)^{-1}}.
\end{split}
\]
Further we have 
\[
\begin{split}
t_{y1}
&=\tr(YM_1)\\
&=s\gamma+\varepsilon+u\delta+v\zeta\\
&=s\gamma+\varepsilon+u\delta+v(t_1-\gamma)\\
&=(s-v)\gamma+u\delta+\varepsilon+v t_1,
\end{split}
\]
then $u\delta+\varepsilon=t_{y1}-v t_1-(s-v)\gamma$ holds. 

On the other hand, 
$\gamma\zeta-\delta\varepsilon=1$ implies $(u\delta)\varepsilon=u(\gamma\zeta-1)$. 
Hence 
$u\delta$ and $\varepsilon$ are the solutions $\theta_1,\theta_2$ of 
the quadratic equation
$$
Z^2-(t_{y1}-v t_1-(s-v)\gamma)Z+u(\gamma\zeta-1)=0
$$
whose coefficients are functions of $t_x,t_y,t_1,t_{xy},t_{x1}$ and $t_{y1}$. 
For any values of $t_x,t_y,t_1,t_{xy},t_{x1}$ and $t_{y1}$, 
there exists at most two solutions $\theta_1,\theta_2$. 
Therefore these six coordinates determine two triples 
$(X,Y,M_1)$ and $(X,Y,M_1')$, that is, 
\[
\begin{split}
X& =
\begin{pmatrix}
\alpha(a)&0\\
0&\alpha(a)^{-1}
\end{pmatrix},~\alpha(a)\not=\pm1;~\\
Y &=
\begin{pmatrix}
s&1\\
u&v
\end{pmatrix},~
u=sv-1\not=0;~\\
M_1 &=
\begin{pmatrix}
\gamma&\theta_1/u\\
\theta_2&\zeta
\end{pmatrix}~
\mathrm{or}~
\begin{pmatrix}
\gamma&\theta_2/u\\
\theta_1&\zeta
\end{pmatrix}.
\end{split}
\]
These two triples yield the two possible values of the coordinate $t_{xy1}=\tr(XYM_1)$: 
$$
t_{xy1}=
\alpha(a) s\gamma+\alpha(a)^{\pm 1}\theta_1+\alpha(a)^{\mp 1}\theta_2+\alpha(a)^{-1} v\zeta. 
$$
In other words, if we also fix the value of $t_{xy1}$, the triple $(X,Y,M_1)$ is 
uniquely determined. Note that, by writing in terms of traces,  $t_{xy1}$ satisfies the following quadratic equation
\begin{eqnarray}
&& t^2_{xy1} - (t_x t_{y1} + t_y t_{x1} + t_1 t_{xy} - t_x t_y t_1) t_{xy1} \label{quad 1}\\
&& + \,  t_x^2+t_y^2 + t_1^2 + t_{xy}^2 + t_{x1}^2 + t_{y1}^2 + t_{xy} t_{x1} t_{y1} - t_x t_y t_{xy} - t_x t_1 t_{x1} - t_y t_1 t_{y1} - 4 =0. \nonumber
\end{eqnarray}
See \cite[Section~5]{GM93-1}. 

Combining all that we have obtained at this point, we have
$$
X=
\begin{pmatrix}
\alpha(a)&0\\
0&\alpha(a)^{-1}
\end{pmatrix},\quad
Y=
\begin{pmatrix}
s&1\\
u& t_y - s
\end{pmatrix},
$$
$$
\alpha(a)^{\pm1}=
\frac{t_x\pm\sqrt{t_x^2-4}}{2},\quad
s=\frac{t_{x y}-\alpha(a)^{-1}t_y}{\alpha(a)-\alpha(a)^{-1}}, \quad
u= - \frac{t_{xy}^2- t_{xy} t_x t_y + t_y^2 + t_x^2-4}{t_x^2-4},
$$
$$
M_1=
\begin{pmatrix}
\gamma&\delta\\
\varepsilon&t_1-\gamma
\end{pmatrix},\quad
\gamma=
\frac{t_{x1}-\alpha(a)^{-1}t_1}{\alpha(a)-\alpha(a)^{-1}}
$$
and
$$
\varepsilon(u\delta)=u(\gamma t_1-\gamma^2-1),\quad
\varepsilon+u\delta
=
t_{y1}-(t_y-s)t_1-\gamma(2s-t_y).
$$
Fixing the value of $t_{xy1}$, the matrix $M_1$ and hence 
the triple $(X,Y,M_1)$ is perfectly determined. 

As a conclusion, in Case 1.1, 
a representation $\rho:G(L)\to \SL$, i.e. 
the triple $(X,Y,M_1)$ of matrices is well determined by the coordinates 
\[
\big(
t_x=2\cos(\pi a/p),t_y=2\cos(\pi b/q),
t_1,
t_{xy},
t_{x1},
t_{y1},
t_{xy1} 
\big)
\]
where $t_{xy}^2- t_{xy} t_x t_y + t_y^2 + t_x^2-4 \not= 0$ and $(t_1,
t_{xy},
t_{x1},
t_{y1},
t_{xy1})$ satisfies equation \eqref{quad 1}. Here $0 <  a < p$, $0 < b < q$, and $a\equiv b\text{ mod }2$. 

So there are 
$\frac{(p-1)(q-1)}{2}$ components, 
which are labeled by $t_x,t_y$. 
Further 
each of them has 
the complex dimension $4$ parametrized by $t_1,t_{xy},t_{x1},t_{y1}, t_{xy1}$ which satisfy equation \eqref{quad 1}. 

Note the followings:
\begin{itemize}
\item 
Any point with $t_{xy}^2- t_{xy} t_x t_y + t_y^2 + t_x^2-4 =0$ (i.e. $t_{xy}=\alpha(a)\beta(b)^{\pm 1}+\alpha(a)^{-1}\beta(b)^{\mp 1}$) belongs to one of the spaces in Case 1.2 below.
\item
$t_{xy1}$ can have only two values after fixing other coordinates $t_x,t_y,t_1,t_{xy},t_{x1}$ and $t_{y1}$. 
\end{itemize}

\smallskip
{\bf Case 1.2.} $u=0$. This is equivalent to $t_{xy}^2- t_{xy} t_x t_y + t_y^2 + t_x^2-4 =0$. 
In this case $\rho|_{G(p,q)}$ is \textit{reducible}. 
Up to conjugation, we may assume that 
$$
X=
\begin{pmatrix}
\alpha(a)&0\\
0&\alpha(a)^{-1}
\end{pmatrix}\quad\mathrm{and}\quad
Y=
\begin{pmatrix}
\beta(b)^{\pm 1}&1\\
0&\beta(b)^{\mp 1}
\end{pmatrix}.
$$
The matrices $X$ and $Y$ are determined by $t_x,t_y$ and $t_{xy}=\alpha(a)\beta(b)^{\pm 1}+\alpha(a)^{-1}\beta(b)^{\mp 1}$.  Moreover, if we fix $t_x$ and $t_y$, there are 2 choices for $t_{xy}$.

For $M_1=\begin{pmatrix}
\gamma&\delta\\
\varepsilon&\zeta
\end{pmatrix}$, 
we can see 
\[
\begin{split}
\gamma
&=
\frac{t_{x1}-\alpha(a)^{-1}t_1}{\alpha(a)-\alpha(a)^{-1}}, \\
\zeta
&=t_1-\gamma\\
&=-\frac{t_{x1}-\alpha(a)t_1}{\alpha(a)-\alpha(a)^{-1}}
\end{split}
\]
as in Case 1.1. 
Since  
\[
\begin{split}
t_{y1}
&=\tr(YM_1)\\
&=\beta(b)^{\pm 1}\gamma+\varepsilon+\beta(b)^{\mp 1}\zeta\\
\end{split}
\]
we can see 
\[
\begin{split}
\varepsilon
&=t_{y1}-\beta(b)^{\pm 1}\gamma-\beta(b)^{\mp 1}\zeta\\
&=\frac{(\alpha(a)-\alpha(a)^{-1})t_{y1}-\beta(b)^{\pm 1}(t_{x1}-\alpha(a)^{-1}t_1)-\beta(b)^{\mp 1}(-t_{x1}+\alpha(a)t_1)}{\alpha(a)-\alpha(a)^{-1}}\\
&=\frac{(\alpha(a)-\alpha(a)^{-1})t_{y1}
-(\beta(b)^{\pm 1}-\beta(b)^{\mp 1})t_{x1}
-(\alpha(a)\beta(b)^{\mp 1}-\alpha(a)^{-1}\beta(b)^{\pm 1})t_1}{\alpha(a)-\alpha(a)^{-1}}.
\end{split}
\]

In this case, the representation $\rho:G(L)\to\SL$ is irreducible if and only if $\varepsilon\neq 0$. 
This open condition is 
\begin{equation} \label{irred 1.2}
(\alpha(a)-\alpha(a)^{-1})t_{y1}-(\beta(b)^{\pm 1}-\beta(b)^{\mp 1})t_{x1}-(\alpha(a)\beta(b)^{\mp 1}-\alpha(a)^{-1}\beta(b)^{\pm 1})t_1\neq 0. 
\end{equation}
Hence there are $\frac{(p-1)(q-1)}{2} \times 2 = (p-1)(q-1)$ components, and each of them has the complex dimension 
$3$ parametrized by $t_{1},t_{x1},t_{y1}$ which satisfy condition \eqref{irred 1.2}.  

\smallskip
\noindent
{\bf Case 2.} $\rho|_{G(p,q)}$ is \textit{abelian}. In this case, 
$(X,Y)$ is conjugate in $\SL$ to a pair of diagonal matrices, 
so it is enough to consider the following cases. 

\smallskip
{\bf Case 2.1.} $\rho|_{G(p,q)} \not= \{\pm I\}$.  Then $X \not= \pm I$ or $Y \not= \pm I$. Without loss of generality, we assume $X \not= \pm I$. 
As in Case 1.2, we have 
\[
\begin{split}
\gamma
&=
\frac{t_{x1}-\alpha(a)^{-1}t_1}{\alpha(a)-\alpha(a)^{-1}}, \\
\zeta
&=-\frac{t_{x1}-\alpha(a)t_1}{\alpha(a)-\alpha(a)^{-1}},\\
t_{y1}
&=\beta(b)^{\pm 1}\gamma+\beta(b)^{\mp 1}\zeta\\
&=\frac{(\beta(b)^{\pm 1}-\beta(b)^{\mp 1})t_{x1}+(\beta(b)^{\mp 1}\alpha(a)-\beta(b)^{\pm 1}\alpha(a)^{-1}))t_1}{\alpha(a)-\alpha(a)^{-1}}.
\end{split}
\]
We then see that 
the representation $\rho$ is irreducible if and only if   $\delta\neq 0$ and $\varepsilon\neq 0$. These conditions  are equivalent to $t_1 \not= \pm 2$ and $\gamma\zeta\neq 1$.  
Here 
\[
\gamma\zeta=-\frac{t_{x1}^2-t_{x1} t_x t_1+t_1^2}{(\alpha(a)-\alpha(a)^{-1})^2}
\]
and then the conditions are 
\begin{equation} \label{irred 2.1}
t_1 \not= \pm 2 \quad \text{and} \quad t_{x1}^2-t_{x1} t_xt_1+t_1^2+t_x^2 -4 \neq 0.
\end{equation}
In this case each component has the dimension $2$ parametrized by $t_1$ and $t_{x1}$ 
which satisfy condition \eqref{irred 2.1}.


\smallskip
{\bf Case 2.2.} $\rho|_{G(p,q)} = \{\pm I\}$. 
In this case $\rho$ is an abelian representation. This is a contradiction.

\subsubsection{$\mu\geq 3$} In this case we are able to explicitly describe the irreducible representations in the following cases. The other cases remain unknown. 

\smallskip

{\bf Case 1.1.} $\rho|_{G(p,q)}$ is \textit{irreducible}. In this case, similar to the previous subsection, 
the representation $\rho:G(L)\to \SL$, i.e. 
the $(\mu+1)$-tuple $(X,Y,M_1,\ldots,M_{\mu-1})$ is 
 well determined by the coordinates 
$$
\big(t_x=2\cos(\pi a/p),t_y=2\cos(\pi b/q),  t_{xy}, t_i, t_{xi},t_{yi}, t_{xyi}\big)_{1 \le i \le \mu -1}
$$
where $t_{xy}^2 - t_x t_y t_{xy} + t_x^2 + t_y^2 -4 \not= 0$ and each quadruple $(t_i, t_{xi},t_{yi}, t_{xyi})$ satisfies the following quadratic equation 
\begin{eqnarray}
&& t^2_{xyi} - (t_x t_{yi} + t_y t_{xi} + t_i t_{xy} - t_x t_y t_i) t_{xyi} \label{quad}\\
&& + \,  t_x^2+t_y^2 + t_i^2 + t_{xy}^2 + t_{xi}^2 + t_{yi}^2 + t_{xy} t_{xi} t_{yi} - t_x t_y t_{xy} - t_x t_i t_{xi} - t_y t_i t_{yi} - 4 =0. \nonumber
\end{eqnarray} 
The coordinates $t_x$ and $t_y$ are determined by $a~(0<a<p)$ and $b~(0<b<q)$ which satisfy 
$a\equiv b~(\mathrm{mod}~2)$. So there are $\frac{(p-1)(q-1)}{2}$ components, 
which are labeled by $t_x,t_y$. 
Further 
each of them has 
the complex dimension $1 + 4 (\mu-1) - (\mu-1) = 3\mu-2$ parametrized by $t_i,t_{xy},t_{xi},t_{yi}, t_{xyi}$, which  satisfy equation \eqref{quad}. 

\smallskip
{\bf Case 2.2.}  $\rho|_{G(p,q)} = \{\pm I\}$.  
In this case $\mathcal{X}(\mu p,\mu q)$ 
can be identified with the irreducible character variety of a free group 
$F_{\mu-1}$. Hence the dimension of each component 
is given by $3(\mu-1)-3=3\mu-6$. See \cite[Section~5]{GM93-1} for details. 

\begin{problem}
Describe $\mathcal{X}(\mu p,\mu q)$ when $\rho|_{G(p,q)}$ is reducible and nontrivial. 
\end{problem}

\section{Twisted Alexander polynomials of torus links}\label{section:4}

In this section, we give an explicit formula for the twisted Alexander polynomial of any torus link 
associated to an irreducible $\SL$-representation. 

Let $\rho\in\mathcal{R}(\mu p,\mu q)$, namely 
$\rho:G(L)\rightarrow \SL$ is an irreducible representation. 
By Lemma~\ref{lem:center}, 
all eigenvalues of $X$ and $Y$ are roots of unity and we may assume that  
$X$ is conjugate to 
$\begin{pmatrix}
\alpha & 0\\
0 & \alpha^{-1}
\end{pmatrix}$ 
and $Y$ is to 
$
\begin{pmatrix}
\beta & 0\\
0 & \beta^{-1}
\end{pmatrix}
$ where 
$\alpha=\alpha(a)=\exp(\sqrt{-1}\pi a/p),\,0\leq a\leq p$ and 
$\beta=\beta(b)=\exp(\sqrt{-1}\pi b/q),\,0\leq b\leq q$.

Recall that $G(L)$ has the presentation 
\[
G(L) = 
\la m_1, \ldots,m_{\mu-1},x, y \mid [x^p,m_i] ~(1\le i \le \mu-1), x^p = y^q \ra .
\]
We put relators $r_i=x^p m_i x^{-p} m_i^{-1}~(1\leq i\leq \mu-1)$ and $r_\mu = x^{p}y^{-q}$. 
Applying the free differential to each $r_i\ (1\leq i\leq \mu-1)$, 
we have
\[
\begin{split}
\frac{\p r_i}{\p m_j}
&=
\begin{cases}
x^p-r_i\ (i=j)\\
0\ (i\neq j),
\end{cases}\\
\frac{\p r_i}{\p x}
&=1+x+\cdots+x^{p-1}-x^pm_ix^{-p} (1+x+\cdots +x^{p-1}) ,\quad
\frac{\p r_i}{\p y}=0.
\end{split}
\]
Moreover 
\[
\frac{\p r_\mu}{\p m_i} =0, ~
\frac{\p r_\mu}{\p x} = 1+x+\cdots +x^{p-1},~
\frac{\p r_\mu}{\p y} = -x^p(1+y+\cdots +y^{q-1}).
\]
We consider the square matrix $A_{\mu+1}$ which is obtained from the Alexander matrix $A$ by 
removing the $(\mu+1)$th column. By the definition of the twisted Alexander polynomial, 
we have 
\[
\begin{split}
&\Delta_{L,\rho}(t_1,\ldots,t_\mu)
=\frac{\det A_{\mu+1}}{\det \Phi(y-1)}\\
&=\frac{
\big( \det (t^{pq} X^p - I) \big)^{\mu-1} \det (I+t^q X+\cdots+t^{q(p-1)}X^{p-1})
}
{\det(t^p Y-I)
}\\
&=\frac{
(t^{2pq}-(\alpha^p+\alpha^{-p})t^{pq}+1)^{\mu-1}
(1+\alpha t^q+\cdots\alpha^{p-1}t^{q(p-1)})(1+\alpha^{-1} t^q+\cdots\alpha^{-(p-1)}t^{q(p-1)})
}
{t^{2p}-(\beta+\beta^{-1})t^p+1
}\\
&=\frac{\big(t^{pq}-(-1)^a\big)^{2\mu}
}
{(t^p-\beta)(t^p-\beta^{-1})(t^q-\alpha)(t^q-\alpha^{-1})
},
\end{split}
\]
where $\alpha^p=(-1)^a$ and we put $t=t_1\cdots t_\mu$. 

It is clear that $\Delta_{L,\rho}(t_1,\ldots,t_\mu)$ depends only on the eigenvalues of $X$ and $Y$, 
which does not vary continuously. Therefore every coefficient of 
$\Delta_{L,\rho}(t_1,\ldots,t_\mu)$ is locally constant. 

To sum up we have the following. 

\begin{theorem}\label{theorem:main}
The twisted Alexander polynomial of the torus link $L=T(\mu p,\mu q)$ is given by 
$$
\Delta_{L,\rho}(t_1,\ldots,t_\mu)
=
\frac{\big(t^{pq}-(-1)^a\big)^{2\mu}
}
{(t^p-\beta)(t^p-\beta^{-1})(t^q-\alpha)(t^q-\alpha^{-1})
},
$$
where $t=t_1\cdots t_\mu$. 
Moreover every coefficient of $\Delta_{L,\rho}(t_1,\ldots,t_\mu)$ is 
locally constant on the irreducible character variety $\mathcal{X}(\mu p,\mu q)$. 
\end{theorem}

\begin{remark}
If $\mu=1$, that is, $L$ is the torus knot $T(p,q)$, 
then the above formula gives the same one as in \cite[Section~4]{KtM12-1} 
and \cite[Theorem~5.16]{Morifuji15-1}.
\end{remark}

\section{Higher dimensional twisted Alexander polynomials}\label{section:higher}

In this section 
we investigate the higher dimensional twisted Alexander polynomial associated to 
an irreducible representation of $\SL$. 

\subsection{Irreducible representations of $\SL$}\label{subsection:5.1}
The group $\SL$ acts naturally on the vector space $\BC^2$. Then the symmetric product $\text{Sym}^{n-1}(\BC^2)$ 
and the induced action by $\SL$ gives an $n$-dimensional irreducible representation of $\SL$. In fact, 
$\text{Sym}^{n-1}(\BC^2)$ can be identified with the vector space $V_n$ of homogeneous polynomials on $\BC^2$ with degree $n-1$, i.e. 
$$V_n = \text{span}_{\BC} \la z_1^{n-1}, z_1^{n-2} z_2, \cdots, z_1 z_2^{n-2}, z_2^{n-1}\ra.$$
The action of $P \in \SL$ on $V_n$ is expressed as
$$P \cdot p(z_1, z_2) := p (P^{-1} \begin{pmatrix}
z_1\\
z_2
\end{pmatrix}).$$
This action defines a representation $\sigma_n: \SL \to \mathrm{GL}(V_n)$. 

It is known that the image of $\sigma_n$ is actually contained in $\mathrm{SL}(n,\BC)$, 
and every $n$-dimensional irreducible representation of $\SL$ is equivalent to $(V_n, \sigma_n)$. 
For a representation $\rho: G(L) \to \SL$, 
we denote the composition $\sigma_n \circ \rho: G(L) \to \mathrm{SL}(n,\BC)$ 
by $\rho_n$. 
Note that $\rho_2 = \rho$ holds.

We now study the twisted Alexander polynomial $\Delta_{L,\rho_n}(t_1, \ldots, t_\mu)$ 
for the torus link $L=T(\mu p,\mu q)$. Let $\rho: G(L) \to \SL$ be an irreducible representation. 
Assume that  
$X$ is conjugate to 
$\begin{pmatrix}
\alpha & 0\\
0 & \alpha^{-1}
\end{pmatrix}$ 
and $Y$ is to 
$
\begin{pmatrix}
\beta & 0\\
0 & \beta^{-1}
\end{pmatrix}
$ where 
$\alpha=\exp(\sqrt{-1}\pi a/p),\,0\leq a\leq p$ and $\beta=\exp(\sqrt{-1}\pi b/q),\,0\leq b\leq q$ 
as before. 
We also put $t=t_1\cdots t_\mu$. 

\begin{theorem} \label{main}
If $n$ is even, then 
$$\Delta_{L,\rho_n} = \frac{\big(t^{pq} - (-1)^a\big)^{n\mu}} { \prod_{j=0}^{\frac{n}{2}-1} \left( t^{2q} - 2 (\cos \frac{(2j+1)a \pi}{p}) t^q + 1 \right) \left( t^{2p} - 2 (\cos \frac{(2j+1)b \pi}{q}) t^p + 1  \right)}.$$

If $n$ is odd, then 
$$\Delta_{L,\rho_n} = \frac{(t^{pq} - 1)^{n\mu}} {(t^p - 1)(t^q-1) \prod_{j=1}^{\frac{n-1}{2}} \left( t^{2q} - 2 (\cos \frac{2ja \pi}{p}) t^q + 1 \right) \left( t^{2p} - 2 (\cos \frac{2jb \pi}{q}) t^p + 1  \right)}.$$
\end{theorem} 

We give a proof of the theorem in the next subsection. 
As an immediate corollary, we have the following. 

\begin{corollary}
Every coefficient of $\Delta_{L,\rho_n}(t_1,\ldots,t_\mu)$ is locally constant 
on the irreducible character variety $\mathcal{X}(\mu p,\mu q)$. 
\end{corollary}

For a link $L$ in $S^3$, by \cite{Ki}, we have $\BT_{L,\rho_n} = \Delta_{L,\rho_n} (1, \ldots, 1)$, 
where $\BT_{L,\rho_n}$ is the Reidemeister torsion (see \cite{Milnor66-1} for the definition). 
Hence, for the torus link $L=T(\mu p,\mu q)$, 
we obtain the following.

\begin{corollary}
Suppose $n$ is even and $a \equiv b \equiv 1 \pmod{2}$. Then  
$$\BT_{L,\rho_n}=\frac{2^{n(\mu-2)}} { \prod_{j=0}^{\frac{n}{2}-1} \sin^2 \left(\frac{(2j+1)a \pi}{2p}\right) \sin^2 \left(\frac{(2j+1)b \pi}{2q}\right)}.$$
\end{corollary}

\begin{remark} For the torus knot $T(p,q)$ (i.e. $\mu=1$), the above formula give 
the same one as in \cite[Proposition~4.1]{Ya13}. 
Moreover, by a similar argument to \cite[Proposition~3.8]{Ya17}, one can show that
$$\lim_{n \to \infty} \frac{\log \BT_{L,\rho_n}}{n} = \left( \mu - \frac{1}{p'} - \frac{1}{q'} \right) \log 2,$$
where $p' = \frac{p}{(a,p)}$ and $q' = \frac{q}{(b,q)}$.
\end{remark}

\subsection{Proof of Theorem \ref{main}}\label{subsection:proof}
Recall 
$ \rho(x) = \begin{pmatrix}\alpha&0\\0&\alpha^{-1}\end{pmatrix}$ 
and $\rho(y)$ is conjugate to $\begin{pmatrix}\beta&0\\0&\beta^{-1}\end{pmatrix}$, 
where 
$\alpha=\exp(\sqrt{-1}\pi a/p),\,0\leq a\leq p$ and $\beta=\exp(\sqrt{-1}\pi b/q),\,0\leq b\leq q$. 
Then it is easy to check that
$\rho_n(x) =\text{diag}(\alpha^{n-1}, \alpha^{n-3}, \ldots, \alpha^{-(n-1)})$ and $\rho_n(y)$ is conjugate to $\text{diag}(\beta^{n-1}, \beta^{n-3}, \ldots, \beta^{-(n-1)})$, 
which are diagonal matrices of degree $n$. Hence
$$
\Delta_{L,\rho_n} =\big( \det (t^{pq} X^p - I) \big)^{\mu-1} \, \frac{
\det (I+t^q X+\cdots+t^{q(p-1)}X^{p-1})
}
{\det(t^p Y-I)
}.$$

If $n$ is even, then
$$\Delta_{L,\rho_n} = \prod_{j=-\frac{n}{2}}^{\frac{n}{2}-1} (\alpha^{(2j+1)p} t^{pq} -1 )^{\mu-1} \frac{1+\alpha^{2j+1} t^q + \cdots + \alpha^{(2j+1)(p-1)} t^{q(p-1)}}{t^{p}-\beta^{2j+1}}.$$
Since $\alpha^p=(-1)^a$, we have $$1+\alpha^{2j+1} t^q + \cdots + \alpha^{(2j+1)(p-1)} t^{(p-1)q} = \frac{\alpha^{(2j+1)p} t^{pq} - 1}{\alpha^{2j+1} t^q - 1} = \frac{(-1)^a t^{pq} - 1}{\alpha^{2j+1} t^q - 1}.$$
Hence 
\begin{eqnarray*} 
\Delta_{L,\rho_n} &=& \prod_{j=-\frac{n}{2}}^{\frac{n}{2}-1} \frac{\big( (-1)^a t^{pq} - 1 \big)^\mu}{\alpha^{2j+1} t^q - 1} \cdot \frac{1}{t^{p}-\beta^{2j+1}}\\
&=& \prod_{j=0}^{\frac{n}{2}-1} \frac{\big( (-1)^a t^{pq} - 1 \big)^\mu}{\alpha^{2j+1} t^q - 1} \cdot \frac{\big( (-1)^a t^{pq} - 1 \big)^\mu}{\alpha^{-(2j+1)} t^q - 1} \cdot \frac{1}{t^{p}-\beta^{2j+1}} \cdot \frac{1}{t^{p}-\beta^{-(2j+1)}}\\
&=& \prod_{j=0}^{\frac{n}{2}-1} \frac{ (t^{pq} - (-1)^a)^{2\mu}}{t^{2q} - (\alpha^{2j+1} + \alpha^{-(2j+1)}) t^q + 1} \cdot \frac{1}{t^{2p} - (\beta^{2j+1} + \beta^{-(2j+1)}) t^p + 1}\\
&=& \prod_{j=0}^{\frac{n}{2}-1} \frac{ (t^{pq} - (-1)^a)^{2\mu}}{t^{2q} - 2 (\cos \frac{(2j+1)a \pi}{p}) t^q + 1} \cdot \frac{1}{t^{2p} - 2 (\cos \frac{(2j+1)b \pi}{q}) t^p + 1}\\
&=& \frac{(t^{pq} - (-1)^a)^{n\mu}} { \prod_{j=0}^{\frac{n}{2}-1} \left( t^{2q} - 2 (\cos \frac{(2j+1)a \pi}{p}) t^q + 1 \right) \left( t^{2p} - 2 (\cos \frac{(2j+1)b \pi}{q}) t^p + 1  \right)}. 
\end{eqnarray*} 

Similarly, if $n$ is odd, then 
\begin{eqnarray*}
\Delta_{L,\rho_n}
&=& \prod_{j=-\frac{n-1}{2}}^{\frac{n-1}{2}} (\alpha^{2jp} t^{pq} -1 )^{n-1} \frac{1+\alpha^{2j} t^q + \cdots + \alpha^{2j(p-1)} t^{q(p-1)}}{t^{p}-\beta^{2j}} \\
&=& \prod_{j=-\frac{n-1}{2}}^{\frac{n-1}{2}} \frac{(t^{pq} - 1)^\mu}{\alpha^{2j} t^q - 1} \cdot \frac{1}{t^{p}-\beta^{2j}} \\
&=& \frac{(t^{pq} - 1)^n}{(t^p - 1)(t^q-1)} \prod_{j=1}^{\frac{n-1}{2}} \frac{(t^{pq} - 1)^{2\mu}}{(\alpha^{2j} t^q - 1)(\alpha^{-2j} t^q - 1)} \cdot \frac{1}{(t^{p}-\beta^{2j})(t^{p}-\beta^{-2j})} \\
&=& \frac{(t^{pq} - 1)^{n\mu}} {(t^p - 1)(t^q-1) \prod_{j=1}^{\frac{n-1}{2}} \left( t^{2q} - 2 (\cos \frac{2ja \pi}{p}) t^q + 1 \right) \left( t^{2p} - 2 (\cos \frac{2jb \pi}{q}) t^p + 1  \right)}. 
\end{eqnarray*} 

This completes the proof of Theorem~\ref{main}.

\subsection*{Acknowledgements} 
The first and second authors have been partially supported by JSPS KAKENHI Grant Numbers 
16K05161 and 17K05261 respectively. 
The third author has been partially supported by a grant from the Simons Foundation (\#354595 to AT).


\end{document}